\documentclass[12pt,reqno]{amsart}
\usepackage{indentfirst,enumerate,cite,amssymb,amsmath,amsthm,mathrsfs,dsfont,mathtools}
\usepackage{color}
\usepackage{comment}
\usepackage[colorlinks=true,linkcolor=blue,citecolor=blue]{hyperref}
\usepackage{geometry}
\usepackage[utf8]{inputenc}

\newtheorem{theorem}{Theorem}[section]
\newtheorem{lemma}[theorem]{Lemma}
\newtheorem{proposition}[theorem]{Proposition}

\theoremstyle{definition}

\newtheorem{definition}[theorem]{Definition}

\newtheorem{remark}[theorem]{Remark}

\numberwithin{equation}{section}

\begin{document}

\title{Differential Complexes in Time-Periodic Gelfand--Shilov Spaces}
	

\author[F. de \'Avila Silva]{Fernando de \'Avila Silva}
\address{
	Departamento de Matem\'atica 
	Universidade Federal do Paran\'a  
	CP 19096, CEP 81531-990, Curitiba 
	Brasil}
\email{fernando.avila@ufpr.br}
	
\author[M. Cappiello]{Marco Cappiello}
\address{Dipartimento di Matematica ``Giuseppe Peano'',
	Università degli Studi di Torino,
	Via Carlo Alberto 10, 10123 Torino,
	Italia
	}
\email{marco.cappiello@unito.it}
	
\author[A. Kirilov]{Alexandre Kirilov}
\address{
	Departamento de Matem\'atica 
	Universidade Federal do Paran\'a  
	CP 19096, CEP 81531-990, Curitiba 
	Brasil}
\email{akirilov@ufpr.br}
	
\author[P. M. Tokoro]{Pedro Meyer Tokoro}
\address{
	Programa de P\'os-Gradua\c c\~ao em Matem\'atica 
	Universidade Federal do Paran\'a  
	CP 19096, CEP 81531-990, Curitiba 
	Brasil}
\email{pedro.tokoro@ufpr.br}

\thanks{The first and third authors thank the support provided by the National Council for Scientific and Technological Development - CNPq, Brazil (grants 316850/2021-7 and 402159/2022-5). This study was financed in part by Capes - Brasil (Finance Code 001). The second author is supported by the Italian Ministry of the University and Research - MUR,	within the PRIN 2022 Call (Project Code 2022HCLAZ8, CUP D53C24003370006).}



	\maketitle 

	\begin{abstract}
		We study the global solvability of a class of differential complexes on the product manifold $\mathbb{T}^m \times \mathbb{R}^n$ associated with systems of 	evolution operators of the form \(L_r = \partial_{t_r} + ia_r(t)P(x,D_x), r=1,\ldots,m,\) where the coefficients $a_r$ are real-valued Gevrey functions on the torus and $P(x,D_x)$ is a globally elliptic normal differential operator on $\mathbb{R}^n$. Within the framework of time-periodic Gelfand--Shilov spaces, we introduce a natural differential complex generated by these operators and investigate its solvability in both functional and ultradistributional settings.
		
		We provide a complete characterization of global solvability in terms of a Diophantine condition involving the constant part of the associated $1$-form and the spectrum of $P$. We also analyze global hypoellipticity of the complex. These results extend previous works on scalar operators and constant coefficient systems to the setting of differential complexes with time-dependent real coefficients.
	\end{abstract}

\section{Introduction}

The study of global hypoellipticity and global solvability for evolution equations and systems on compact manifolds has attracted considerable  attention in recent years, particularly in connection with arithmetic conditions on the coefficients and the spectral properties of the
underlying operators. In this paper, we investigate these properties for a class of differential
complexes on the product manifold $\mathbb{T}^m \times \mathbb{R}^n$, generated by first-order evolution
operators of the form
\begin{equation}\label{generalsystem}	
	L_r = \partial_{t_r} + ia_r(t)P(x,D_x),
	\qquad r=1,\ldots,m,
\end{equation}
where $t=(t_1,\ldots,t_m)\in\mathbb{T}^m$ and $x\in\mathbb{R}^n$.

Throughout the paper, the coefficients $a_r$ are assumed to be real-valued
functions belonging to a Gevrey class of order $\sigma>1$ on the torus
$\mathbb{T}^m$.
We further assume that the $1$-form
\[
a(t)=\sum_{r=1}^m a_r(t)\mathrm{d}t_r
\]
is closed.
Under this assumption, $a$ admits the decomposition
\[
a(t)=a_0+\mathrm{d}_t A(t),
\]
where $a_0=\sum_{r=1}^m a_{r,0}\mathrm{d}t_r$ is a constant $1$-form, whose
coefficients are given by
\begin{equation*}
	a_{r,0}=\frac{1}{2\pi}\int_0^{2\pi} a_r(0,\ldots,t_r,\ldots,0)\mathrm{d}t_r,
\end{equation*}
$\mathrm{d}_t$ denotes the exterior derivative on $\mathbb{T}^m$, and
$A\in\mathcal{G}^\sigma(\mathbb{T}^m)$ is a real-valued Gevrey function.

The operator $P(x,D_x)$ in \eqref{generalsystem} is assumed to be a normal
differential operator of the form
\begin{equation}\label{P-intro}
	P(x,D_x)=\sum_{|\alpha|+|\beta|\le M}
	c_{\alpha,\beta} x^\beta \partial_x^\alpha,
	\qquad c_{\alpha,\beta}\in\mathbb{C},
\end{equation}
of order $M\ge2$, satisfying the global ellipticity condition
\begin{equation}\label{P-elliptic}
	p_M(x,\xi)=\sum_{|\alpha|+|\beta|=M}
	c_{\alpha,\beta}\,x^\beta\,\xi^\alpha\neq0,
	\qquad (x,\xi)\neq(0,0).
\end{equation}

Under these assumptions, $P$ has a discrete real spectrum given by a sequence of eigenvalues
$\{\lambda_j\}_{j\in\mathbb{N}}$, with $|\lambda_j|\to\infty$ as
$j\to\infty$, and exhibits the asymptotic behavior
$|\lambda_j|\sim j^{M/(2n)}$.
Moreover, the associated eigenfunctions $\{\varphi_j\}$ form an orthonormal
basis of $L^2(\mathbb{R}^n)$ and belong to the Gelfand--Shilov space
$\mathcal{S}^{1/2}_{1/2}(\mathbb{R}^n)$.

Gelfand--Shilov spaces $\mathcal{S}^\mu_\nu(\mathbb{R}^n)$ provide a natural
functional framework for the analysis of equations involving globally
elliptic operators such as \eqref{P-intro}--\eqref{P-elliptic}
(see, e.g., \cite{CGR1,CGR2,Nicola-Rodino}).
Motivated by the study of evolution equations with periodic time
dependence, we work instead with their time-periodic counterparts,
introduced in \cite{AviCap22,AviCapKir25}.
For fixed $\sigma>1$ and $\mu\ge1/2$, we denote by
$\mathcal{S}_{\sigma,\mu,C}$ the Banach space of all smooth functions
$u$ on $\mathbb{T}^m\times\mathbb{R}^n$ satisfying
\begin{equation}\label{firstnorm}
	|u|_{\sigma,\mu,C}
	=
	\sup_{\alpha,\beta\in\mathbb{N}_0^n}
	\sup_{\gamma\in\mathbb{N}_0^m}
	C^{-|\alpha+\beta|-|\gamma|}
	\gamma!^{-\sigma}(\alpha!\beta!)^{-\mu}
	\sup_{(t,x)}
	|x^\alpha\partial_x^\beta\partial_t^\gamma u(t,x)|
	<\infty,
\end{equation}
or equivalently,
\begin{equation}\label{secondnorm}
 	\|u\|_{\sigma,\mu,C}:=\sup_{\gamma \in \mathbb{N}^m, M \in \mathbb{N}} C^{-M - |\gamma|} M!^{-m\mu} \gamma!^{-\sigma} \| P^M \partial_t^\gamma u \|_{L^2(\mathbb{T}^m \times \mathbb{R}^n)}<\infty,
\end{equation}
cf. \cite[Remark 2.1]{AviCapKir26}.
The norms $|u|_{\sigma,\mu,C}$ and $\|u\|_{\sigma,\mu,C}$ define equivalent topologies on $\mathcal{S}_{\sigma,\mu,C}$. We then set
$\mathcal{S}_{\sigma,\mu}=\bigcup_{C>0}\mathcal{S}_{\sigma,\mu,C}$
and define
\[
\mathscr{F}_\mu=\bigcup_{\sigma>1}\mathcal{S}_{\sigma,\mu},
\]
endowed with its natural inductive limit topology.
We denote by $\mathscr{F}_\mu'$ its strong dual.

The global hypoellipticity and solvability in $\mathscr{F}_\mu$ of the scalar
operator
\[
L=\partial_t+ia(t)P(x,D_x)
\]
have been completely characterized in
\cite{AviCap22,AviCap24}, while the case of systems with constant coefficients
was treated in \cite{AviCapKir25}.
More recently, variable-coefficient decoupled systems were studied in
\cite{AviCapKir26}.
The present work extends these results to the case of real-valued,
time-dependent coefficients within the framework of differential complexes.

To this end, for each $p=0,\ldots,m-1$, we consider the spaces
\[
\mathscr{F}_{\mu,p}
=\textstyle\bigwedge^{p,0}\mathscr{F}_\mu(\mathbb{T}^m\times\mathbb{R}^n),
\]
together with their ultradistributional counterparts
$\mathscr{F}'_{\mu,p}$.
We define the operators
\[
\mathbb{L}^p u
=\mathrm{d}_t u + ia(t)\wedge P(x,D_x)u,
\]
which satisfy $\mathbb{L}^{p+1}\circ\mathbb{L}^p=0$ and therefore generate
the differential complex
\[
0\to\mathscr{F}_{\mu,0}\xrightarrow{\mathbb{L}^0}
\mathscr{F}_{\mu,1}\xrightarrow{\mathbb{L}^1}\cdots
\xrightarrow{\mathbb{L}^{m-1}}\mathscr{F}_{\mu,m}\to0.
\]

The main goal of this paper is to characterize the global solvability and hypoellipticity of this complex. We show that solvability is governed by Diophantine conditions on the constant part $a_0$, while global hypoellipticity exhibits a sharp contrast between the scalar case $p=0$ and higher degrees $p\ge 1$. The present work is inspired by the seminal contributions of F.~Tr\`eves \cite{Tre76,Treves-book}, the book by S.~Berhanu, P.~Cordaro, and J.~Hounie \cite{BerhCorHou-book}, as well as several works devoted to the periodic case \cite{BerCorMal93,BerCorPet04,BerPet99jmaa}, among many others.

The paper is organized as follows.
In Section~\ref{sec-preliminaries} we recall the necessary background on
Fourier analysis in time-periodic Gelfand--Shilov spaces and introduce the
notions of global solvability and hypoellipticity for differential
complexes.
In Section~\ref{sec-solvability} we obtain a complete characterization of
global solvability (Theorem~\ref{thm_GS}), and in
Section~\ref{sec-hypoellipticity} we study global hypoellipticity,
establishing Theorems~\ref{hypo0} and~\ref{hypo>0}.

\section{Preliminaries} \label{sec-preliminaries}

We begin by recalling the characterization of Gevrey functions and ultradistributions on the torus $\mathbb{T}^m$ in terms of their Fourier coefficients.

For $h > 0$ and $\sigma \geq 1$, we denote by $\mathcal{G}^{\sigma,h}(\mathbb{T}^m)$ the Banach space of all smooth functions $\varphi \in C^\infty(\mathbb{T}^m)$ satisfying
\begin{equation*}
	\sup_{t \in \mathbb{T}^m} |\partial^\gamma \varphi(t)| \leq C h^{|\gamma|} (\gamma!)^{\sigma}, \quad \text{for all } \gamma \in \mathbb{N}_0^m.
\end{equation*}
for some constant $C > 0$. The corresponding norm is given by
\begin{equation*}
	\|\varphi\|_{\sigma,h} \coloneqq \sup_{\gamma \in \mathbb{N}_0^m} \left\{ \sup_{t \in \mathbb{T}^m} |\partial^\gamma \varphi(t)| \, h^{-|\gamma|} (\gamma!)^{-\sigma} \right\}.
\end{equation*}

The space of periodic Gevrey functions of order $\sigma$ is then defined as the inductive limit
\[
\mathcal{G}^{\sigma}(\mathbb{T}^m) =  \underset{h\rightarrow +\infty}{\operatorname{ind} \lim} \ \mathcal{G}^{\sigma, h} (\mathbb{T}^m),
\]
and its topological dual is denoted by $(\mathcal{G}^{\sigma})'(\mathbb{T}^m)$.

The Fourier coefficients of a Gevrey function $f \in \mathcal{G}^{\sigma}(\mathbb{T}^m)$ or of an ultradistribution $f \in (\mathcal{G}^{\sigma})'(\mathbb{T}^m)$ are given by
\[
\widehat{f}(\tau) \coloneqq \frac{1}{(2\pi)^m} \int_{\mathbb{T}^m} f(t) e^{-i \tau \cdot t}  \mathrm{d}t, \quad \tau \in \mathbb{Z}^m,
\]
and, in the distributional sense,
\[
\widehat{f}(\tau) \coloneqq \langle f, e^{-i \tau \cdot t} \rangle, \quad \tau \in \mathbb{Z}^m.
\]

Consequently, any such $f$ admits a Fourier series representation:
\[
f(t) = \sum_{\tau \in \mathbb{Z}^m} \widehat{f}(\tau) e^{i \tau \cdot t}.
\]

\subsection{Eigenfunction Expansions for Time-Periodic Gelfand–Shilov Spaces} \

We now recall the Fourier analysis developed in
\cite{AviCap22,AviCapKir25} for time-periodic Gelfand--Shilov spaces.
 To this end, we review the characterization of the spaces $\mathcal{S}_{\sigma, \mu}$ and $\mathcal{S}'_{\sigma, \mu}$ in terms of eigenfunction expansions.

Let \(\varphi_j \in \mathcal{S}_{1/2}^{1/2}(\mathbb{R}^n)\), \(j \in \mathbb{N}\), be the eigenfunctions of the operator $P$ defined in \eqref{P-intro}. Suppose $u \in \mathcal{S}'_{\sigma, \mu}(\mathbb{T}^m \times \mathbb{R}^n)$. Then, for each $j \in \mathbb{N}$, the linear functional
\begin{equation*}
	\langle u_j(t), \psi(t) \rangle \coloneqq \langle u, \psi(t)\varphi_j(x) \rangle
\end{equation*}
defines an element $u_j \in (\mathcal{G}^\sigma)'(\mathbb{T}^m)$.

Moreover, for every $\varepsilon > 0$ and $h > 0$, there exists a constant $C_{\varepsilon,h} > 0$ such that
\begin{equation*}
	|\langle u_j, \psi \rangle| \leq C_{\varepsilon,h} \| \psi \|_{\sigma, h} \exp\left( \varepsilon j^{\frac{1}{2n\mu}} \right),
\end{equation*}
for all $j \in \mathbb{N}$ and $\psi \in \mathcal{G}^{\sigma,h}(\mathbb{T}^m)$.

As a consequence, $u$ admits the representation
\[
\langle u, \theta \rangle = \sum_{j \in \mathbb{N}} \langle u_j(t)\varphi_j(x), \theta \rangle,
\]
where the pairing is defined by
\[
\langle u_j(t)\varphi_j(x), \theta(t,x) \rangle \coloneqq 
\left\langle u_j(t), \int_{\mathbb{R}^n} \theta(t,x)\varphi_j(x) \, \mathrm{d}x \right\rangle.
\]

Conversely, assume that a sequence $\{u_j\}_{j \in \mathbb{N}}$ of elements of
$(\mathcal{G}^\sigma)'(\mathbb{T}^m)$ satisfies the following condition:
for every $\varepsilon > 0$ and $h > 0$, there exists $C_{\varepsilon,h} > 0$
such that
\begin{equation}\label{estimate-distr}
	|\langle u_j, \psi \rangle| \leq C_{\varepsilon,h} \| \psi \|_{\sigma,h}
	\exp\left( \varepsilon j^{\frac{1}{2n\mu}} \right),
\end{equation}
for all $j \in \mathbb{N}$ and $\psi \in \mathcal{G}^{\sigma,h}(\mathbb{T}^m)$.

Then, the formal series
\[
u(t,x) = \sum_{j \in \mathbb{N}} u_j(t)\varphi_j(x)
\]
defines an element of $\mathcal{S}'_{\sigma,\mu}(\mathbb{T}^m \times \mathbb{R}^n)$, and the coefficients satisfy the identity
\[
\langle u_j, \psi(t) \rangle = \langle u, \psi(t)\varphi_j(x) \rangle,
\]
for every \( \psi \in \mathcal{G}^\sigma(\mathbb{T}^m).\)

In particular, if each $u_j$ belongs to $\mathcal{G}^\sigma(\mathbb{T}^m)$, the condition \eqref{estimate-distr} can be replaced by
\begin{equation*}
	\sup_{t \in \mathbb{T}^m} |u_j(t)| \leq C_\varepsilon \exp\left( \varepsilon j^{\frac{1}{2n\mu}} \right),
\end{equation*}
for all \(j \in \mathbb{N}.\)

\begin{proposition}\label{Thm-estimate-derivatives-functions}
	Let $\mu \geq 1/2$ and $\sigma \geq 1$. Then a distribution $u \in \mathcal{S}'_{\sigma,\mu}(\mathbb{T}^m \times \mathbb{R}^n)$ belongs to $\mathcal{S}_{\sigma,\mu}(\mathbb{T}^m \times \mathbb{R}^n)$ if and only if it admits a decomposition
	\[
	u(t,x) = \sum_{j \in \mathbb{N}} u_j(t) \varphi_j(x),
	\]
	where
	\[
	u_j(t) = \int_{\mathbb{R}^n} u(t,x)\varphi_j(x) \, \mathrm{d}x,
	\]
	and there exist constants $C > 0$ and $\varepsilon > 0$ such that
	\begin{equation} \label{deccoeff}
		\sup_{t \in \mathbb{T}^m} | \partial_t^\gamma u_j(t) | \leq 
		C^{|\gamma| + 1} (\gamma!)^{\sigma} \exp\left( -\varepsilon j^{\frac{1}{2n\mu}} \right),
	\end{equation}
	for all \(j \in \mathbb{N}\) and \( \gamma \in \mathbb{N}_0^m.\)
\end{proposition}

\begin{proof}
	See \cite[Theorem 2.4]{AviCap22}.
\end{proof}

Similarly, we obtain a characterization of time-periodic Gelfand–Shilov functions and ultradistributions in terms of their full Fourier coefficients.

\begin{proposition}\label{charac_full_fourier-functions}
	Let $\{a(\tau,j)\}_{(\tau,j) \in \mathbb{Z}^m \times \mathbb{N}}$ be a sequence of complex numbers, and consider the formal series
	\[
	f(t,x) = \sum_{j \in \mathbb{N}} \sum_{\tau \in \mathbb{Z}^m} a(\tau,j) \, e^{i t \cdot \tau} \varphi_j(x), \quad (t,x) \in \mathbb{T}^m \times \mathbb{R}^n.
	\]
	Then $f \in \mathcal{S}_{\sigma,\mu}$ if and only if there exist constants $\varepsilon > 0$ and $C > 0$ such that
	\begin{equation}\label{seq-full-coeff-funct}
		|a(\tau,j)| \leq C 
		\exp\left[-\varepsilon \left(\|\tau\|^{\frac{1}{\sigma}} + j^{\frac{1}{2n\mu}}\right)\right], \quad  (\tau,j) \in \mathbb{Z}^m \times \mathbb{N}.
	\end{equation}
	
	Moreover, under this condition, the Fourier coefficients satisfy $\widehat{a_j}(\tau) = a(\tau,j)$ for all $(\tau,j)$, where
	\[
	\widehat{a_j}(\tau) = \int_{\mathbb{T}^m} a_j(t) e^{-i t \cdot \tau} \, \mathrm{d}t, \quad \text{with} \quad a_j(t) = \int_{\mathbb{R}^n} a(t,x) \varphi_j(x) \, \mathrm{d}x.
	\]
\end{proposition}

\begin{proof}
	See \cite[Theorem 2.4]{AviCapKir25}.
\end{proof}

Likewise, a similar characterization holds for ultradistributions.

\begin{proposition}\label{charac_full_fourier_ultradistr}
	Let $\{a(\tau,j)\}_{(\tau,j) \in \mathbb{Z}^m \times \mathbb{N}}$ be a sequence of complex numbers, and consider the formal series
	\[
	u(t,x) = \sum_{j \in \mathbb{N}} \sum_{\tau \in \mathbb{Z}^m} a(\tau,j) \, e^{i t \cdot \tau} \varphi_j(x), \quad (t,x)\in\mathbb{T}^m\times\mathbb{R}^n.
	\]
	Then $u \in \mathcal{S}'_{\sigma,\mu}$ if and only if for every $\varepsilon > 0$ there exists $C_\varepsilon > 0$ such that
	\begin{equation}\label{seq-full-coeff-distr}
		|a(\tau,j)| \leq C_\varepsilon 
		\exp\left[\varepsilon \left(\|\tau\|^{\frac{1}{\sigma}} + j^{\frac{1}{2n\mu}}\right)\right], \quad  (\tau,j) \in \mathbb{Z}^m \times \mathbb{N}.
	\end{equation}
	
	Moreover, when this condition holds, we have $a(\tau,j) = \widehat{a_j}(\tau)$, where
	\[
	\langle a_j(t), \psi(t) \rangle \coloneqq \langle a, \psi(t) \varphi_j(x) \rangle, \quad \psi \in \mathcal{G}^\sigma(\mathbb{T}^m).
	\]
\end{proposition}

\begin{proof}
	See \cite[Theorem 2.5]{AviCapKir25}.
\end{proof}

Using these characterizations, we can describe the spaces $\mathscr{F}_{\mu,p}$ and $\mathscr{F}'_{\mu,p}$ for $p = 0, \ldots, m - 1$. Specifically, any $u \in \mathscr{F}_{\mu,p}$ can be written as
\[
u(t,x) = \sum_{|K|=p} u_K(t,x) \, \mathrm{d}t_K = \sum_{j \in \mathbb{N}} u_j(t) \varphi_j(x),
\]
where
\[
u_j(t) = \sum_{|K|=p} u_{K,j}(t) \, \mathrm{d}t_K \in \textstyle\bigwedge^p \mathcal{G}^\sigma(\mathbb{T}^m),
\]
and each $u_{K,j}(t) \in \mathcal{G}^\sigma(\mathbb{T}^m)$ denotes the coefficient in the Fourier–Hermite expansion of $u_K$. A similar decomposition holds for $u \in \mathscr{F}'_{\mu,p}$.

We can interpret $\mathscr{F}_{\mu,p}$ as the space of differential $p$-forms on $\mathbb{T}^m \times \mathbb{R}^n$ whose coefficients belong to the ultradifferentiable class $\mathscr{F}_\mu$. Accordingly, its strong dual $\mathscr{F}'_{\mu,p}$ is identified with the space of $p$-currents, that is, continuous linear functionals acting on compactly supported $p$-forms with coefficients in $\mathscr{F}_\mu$.

Furthermore, the full Fourier series can be written as
\[
u(t,x) = \sum_{j \in \mathbb{N}} \left( \sum_{\tau \in \mathbb{Z}^m} \widehat{u}_j(\tau) \, e^{i \tau \cdot t} \right) \varphi_j(x),
\]
where
\[
\widehat{u}_j(\tau) = \sum_{|K|=p} \widehat{u}_{K,j}(\tau) \, \mathrm{d}t_K,
\]
and $\widehat{u}_{K,j}(\tau)$ are the Fourier coefficients of $u_{K,j}(t)$.

With this structure, the equation $\mathbb{L}^p u = f$ can be reduced to a sequence of equations involving the Fourier coefficients of $u$ and $f$. Specifically, $\mathbb{L}^p u = f$ is equivalent to the family of equations
\[
d_t u_j(t) + i \lambda_j a(t) \wedge u_j(t) = f_j(t), \quad j \in \mathbb{N}.
\]

Moreover, associating the 1-form $a_0$ to the vector $\boldsymbol{a}_0 \coloneqq (a_{1,0}, \ldots, a_{m,0}) \in \mathbb{R}^m$, one can show that
\[
d_t\left[e^{i \lambda_j (\boldsymbol{a}_0 \cdot t + A(t))} u_j(t) \right] = e^{i \lambda_j (\boldsymbol{a}_0 \cdot t + A(t))} f_j(t),
\]
whenever $\lambda_j \boldsymbol{a}_0 \in \mathbb{Z}^m$. In such cases, the $(p+1)$-form $e^{i \lambda_j (\boldsymbol{a}_0 \cdot t + A(t))} f_j(t)$ is exact.

We now define the subspace of data for which solvability will be studied.

\begin{definition}
	Let $p = 0, \ldots, m - 1$. Define the space $\mathbb{E}_\mu^p \subset \mathscr{F}_{\mu,p+1}$ as the set of all $f \in \mathscr{F}_{\mu,p+1}$ such that:
	\[
	\mathbb{L}^{p+1} f = 0, \mbox{ and } e^{i \lambda_j (\boldsymbol{a}_0 \cdot t + A(t))} f_j(t) \mbox{ is exact whenever } \lambda_j \boldsymbol{a}_0 \in \mathbb{Z}^m.
	\]
\end{definition}

\begin{definition}
	Let $p = 0, \ldots, m - 1$. The operator $\mathbb{L}^p$ is said to be:
	\begin{itemize}
		\item Globally solvable if for every $f \in \mathbb{E}_\mu^p$ there exists $u \in \mathscr{F}_{\mu,p}$ such that $\mathbb{L}^p u = f$.
		
		\item Globally hypoelliptic if the conditions $u \in \mathscr{F}'_{\mu,p}$ and $\mathbb{L}^p u \in \mathscr{F}_{\mu,p+1}$ imply $u \in \mathscr{F}_{\mu,p}$.
	\end{itemize}
\end{definition}

\section{Reduction to Normal Form}\label{sec-normalformreduction}

To derive necessary and sufficient conditions for the global hypoellipticity and solvability of the differential complex, we begin by reducing the problem to an equivalent system whose coefficients are independent of the toroidal variable. As before, we write $a = a_0 + \mathrm{d}_t A \in \textstyle\bigwedge^1 \mathcal{G}^\sigma(\mathbb{T}^m)$, where $a_0$ is a constant $1$-form and $A$ is a real-valued Gevrey function.

\begin{theorem}\label{iso_A}
	Let $p = 0, \ldots, m - 1$. The map $\mathscr{T}_A^p$ defined by
	\[
	\sum_{j \in \mathbb{N}} u_j(t)\varphi_j(x) \mapsto \sum_{j \in \mathbb{N}} u_j(t)e^{i\lambda_j A(t)}\varphi_j(x)
	\]
	is a well-defined continuous linear topological isomorphism from $\mathscr{F}_{\mu,p}$ to $\mathscr{F}_{\mu,p}$ and from $\mathscr{F}'_{\mu,p}$ to $\mathscr{F}'_{\mu,p}$ with inverse $\mathscr{T}_{-A}^p$. Moreover, the following conjugation identity holds:
	\[
	\mathscr{T}_A^{p+1} \circ \mathbb{L}^p \circ \mathscr{T}_{-A}^p = \mathbb{L}_0^p,
	\]
	where
	\[
	\mathbb{L}_0^p u = \mathrm{d}_t u + i a_0 \wedge P(x,D_x) u.
	\]
\end{theorem}

\begin{proof}
	We prove the conjugation identity; the other claims follow from \cite[Proposition 3.2]{AviCapKir26}. Let $u \in \mathscr{F}_{\mu,p}$, and write $\mathscr{T}_{-A}^p u = \sum_{j} u_j(t)e^{-i\lambda_j A(t)}\varphi_j(x)$. Then,
	\begin{align*}
		\mathbb{L}^p \left( \mathscr{T}_{-A}^p u \right) 
		&= \sum_{j \in \mathbb{N}} \left[ \mathrm{d}_t \left( u_j(t)e^{-i\lambda_j A(t)} \right) + i \lambda_j a(t) \wedge u_j(t)e^{-i\lambda_j A(t)} \right] \varphi_j(x) \\
		&= \sum_{j \in \mathbb{N}} \left[ \mathrm{d}_t u_j(t) - i \lambda_j \mathrm{d}_t A(t) \wedge u_j(t) + i \lambda_j a(t) \wedge u_j(t) \right] e^{-i\lambda_j A(t)} \varphi_j(x) \\
		&= \sum_{j \in \mathbb{N}} \left[ \mathrm{d}_t u_j(t) + i \lambda_j a_0 \wedge u_j(t) \right] e^{-i\lambda_j A(t)} \varphi_j(x) \\
		&= \mathscr{T}_{-A}^{p+1} \left( \mathbb{L}_0^p u \right).
	\end{align*}
	Multiplying on the left by $\mathscr{T}_A^{p+1}$ yields the desired identity. The extension of $\mathscr{T}_A^{p}$ on $\mathscr{F}'_{\mu,p}$ can be obtained as in \cite[Proposition 3.1]{AviCapKir26} using the characterization of time-periodic Gelfand-Shilov spaces in terms of the norm \eqref{secondnorm}.
\end{proof}

\begin{proposition}\label{equiv_L0}
	Let $p \in \{0, \ldots, m - 1\}$. Then 
the operator $\mathbb{L}^p$ is globally solvable (respectively, globally hypoelliptic) if and only if $\mathbb{L}_0^p$ is globally solvable (respectively, globally hypoelliptic).
\end{proposition}
\begin{proof}
	Concerning global solvability, we recall that $\mathscr{F}_\mu = \mathscr{F}_{\mu,0}$ is a DFS space. Indeed, by \cite[Theorem 2.4]{KowTok}, $\mathcal{S}_{\sigma,\mu}$ is a DFS space, for each $\sigma>1$. Since $\mathscr{F}_\mu$ is the inductive limit of these spaces, by \cite{Komatsu1967}, we conclude that $\mathscr{F}_\mu$ is itself DFS. A similar argument applies to $\mathscr{F}_{\mu,p}$ for each $p = 1, \ldots, m-1$, observing its structure as an inductive limit of Banach spaces.
	
	By a result of Araujo \cite[Lemma 2.2]{Araujo2017}, for continuous linear operators between DFS spaces, global solvability is equivalent to having closed range. Since the conjugation operator $\mathscr{T}_A^p$ is a topological isomorphism, the property of having closed range is preserved under conjugation, transferring the equivalence between $\mathbb{L}^p$ and $\mathbb{L}_0^p$.
	
	Regarding global hypoellipticity, assume first that $\mathbb{L}^p$ is globally hypoelliptic. Let $u \in \mathscr{F}'_{\mu,p}$ be such that $\mathbb{L}_0^p u \in \mathscr{F}_{\mu,p+1}$. From the proof of Theorem \ref{iso_A}, we have the identity
	\[
	\mathbb{L}^p (\mathscr{T}_{-A}^p u) = \mathscr{T}_{-A}^{p+1} (\mathbb{L}_0^p u).
	\]
	Since $\mathbb{L}_0^p u \in \mathscr{F}_{\mu,p+1}$ and $\mathscr{T}_{-A}^{p+1}$ maps $\mathscr{F}_{\mu,p+1}$ into itself, it follows that $\mathbb{L}^p (\mathscr{T}_{-A}^p u) \in \mathscr{F}_{\mu,p+1}$. The global hypoellipticity of $\mathbb{L}^p$ then implies that $\mathscr{T}_{-A}^p u \in \mathscr{F}_{\mu,p}$. Finally, since $\mathscr{T}_{-A}^p$ is an isomorphism, we conclude that $u \in \mathscr{F}_{\mu,p}$. The converse implication is analogous.
\end{proof}

In view of Proposition~\ref{equiv_L0}, in what follows we may assume without loss of generality that $a = a_0$ is constant, and thus $\mathbb{L}^p = \mathbb{L}_0^p$.

\section{Global Solvability}\label{sec-solvability}

Let $u \in \mathscr{F}'_{\mu,p}$ and $f \in \mathscr{F}_{\mu,p+1}$ satisfy  $\mathbb{L}^p u = f$. By expanding $u$ and $f$ in Fourier series, we obtain:

\begin{align*}
	\mathbb{L}^p u &= \sum_{(\tau,j)\in \mathbb{Z}^m \times \mathbb{N}} \left( \sum_{r=1}^m i \tau_r e^{i \tau \cdot t} \varphi_j(x)\, \mathrm{d}t_r \right) \wedge \widehat{u}_{j}(\tau) \\
	&\quad + \sum_{(\tau,j) \in \mathbb{Z}^m \times \mathbb{N}} e^{i \tau \cdot t} \varphi_j(x) \left( \sum_{r=1}^m i \lambda_j a_{r,0} \mathrm{d}t_r \right) \wedge \widehat{u}_j(\tau) \\
	&= \sum_{(\tau,j) \in \mathbb{Z}^m \times \mathbb{N}} e^{i \tau \cdot t} \varphi_j(x) \widehat{f}_j(\tau),
\end{align*}
where $\tau = (\tau_1, \ldots, \tau_m) \in \mathbb{Z}^m$.

Therefore, for each $(\tau,j)$, the Fourier coefficient $\widehat{u}_j(\tau)$ must satisfy the algebraic equation
\begin{equation}\label{a_wedge_u_f}
	a_j(\tau) \wedge \widehat{u}_j(\tau) = \widehat{f}_j(\tau),
\end{equation}
where
\begin{equation}\label{ajtau}
	a_j(\tau) = i \sum_{r=1}^m \left( \tau_r + \lambda_j a_{r,0} \right) \mathrm{d}t_r.
\end{equation}

To simplify notation, we define the norm
\[
\| \tau + \lambda_j \boldsymbol{a}_0 \| \coloneqq \max_{1 \leq r \leq m} |\tau_r + \lambda_j a_{r,0}|.
\]

We recall the following algebraic lemma \cite[Lemma 2.1]{BerPet99jmaa}:

\begin{lemma}\label{sol_ext_alg}
	Let $\mathscr{L} = \sum_{r=1}^m \mathscr{L}_r \, \mathrm{d}t_r \in \textstyle\bigwedge^1 \mathbb{C}^m \setminus \{0\}$ be a nontrivial constant 1-form, and let $\mathscr{H} = \sum_{|J| = p+1} \mathscr{H}_J \, \mathrm{d}t_J \in \textstyle\bigwedge^{p+1} \mathbb{C}^m$, with $p = 0, \ldots, m-1$. Then the equation
	\[
	\mathscr{L} \wedge \mathscr{U} = \mathscr{H}
	\]
	admits a solution $\mathscr{U} \in \textstyle\bigwedge^p \mathbb{C}^m$ if and only if
	\[
	\mathscr{L} \wedge \mathscr{H} = 0.
	\]
	
	Moreover, when this condition holds, a particular solution is given by
	\[
	\mathscr{U}_0 = \sum_{|J|=p+1} \sum_{r' \in J} (-1)^{\mathrm{sign}(r', J)} (\mathscr{L}_{r'})^{-1} \mathscr{H}_J \, \mathrm{d}t_{J \setminus \{r'\}},
	\]
	where the sign is determined by the order of indices and the notation $\widehat{r'}$ indicates omission.
	
	The general solution is then given by
	\[
	\mathscr{U} = \mathscr{U}_0 + \mathscr{L} \wedge \mathscr{W},
	\]
	where $\mathscr{W} \in \textstyle\bigwedge^{p-1} \mathbb{C}^m$ is arbitrary. (By convention, $\textstyle\bigwedge^{-1} \mathbb{C}^m = \{0\}$.)
\end{lemma}

We also introduce the index set
\begin{equation}\label{set_Z}
	\mathcal{Z} \coloneqq \left\{ j \in \mathbb{N} \mid \lambda_j \boldsymbol{a}_0 \in \mathbb{Z}^m \right\},
\end{equation}
which captures the so-called resonant indices for which the vector $\lambda_j \boldsymbol{a}_0$ lies in the integer lattice.

We now define a Diophantine-type condition that plays a central role in our solvability criterion.

\begin{definition}
	Let $\sigma > 1$ and $\mu \geq 1/2$. A vector $\boldsymbol{\alpha} = (\alpha_1, \ldots, \alpha_m) \in \mathbb{R}^m$ is said to satisfy the Diophantine condition ($\mathrm{DC}_{\sigma,\mu}$) with respect to a sequence of real numbers $\{\lambda_j\}_{j \in \mathbb{N}}$ if for every $\varepsilon > 0$, there exists a constant $C_\varepsilon > 0$ such that
	\begin{equation}\label{DC}
		\| \tau + \lambda_j \boldsymbol{\alpha} \| = \max_{1\leq r\leq m}|\tau_r+\lambda_j \alpha_r| \geq C_{\varepsilon} \exp \left(-\varepsilon(\|\tau\|^{1/\sigma} + j^{\frac{1}{2n\mu}})\right), \tag{$\mathrm{DC}_{\sigma,\mu}$}
	\end{equation}
	for all $(\tau, j) \in \mathbb{Z}^m \times \mathbb{N}$ such that $\tau + \lambda_j \boldsymbol{\alpha} \neq 0$.
\end{definition}

We are now ready to state the main result of this section.

\begin{theorem}\label{thm_GS}
	Let $p \in \{0, \ldots, m-1\}$. The operator $\mathbb{L}^p$ is globally solvable if and only if the vector $\boldsymbol{a}_0$ satisfies the condition {\em (}$\mathrm{DC}_{\sigma,\mu}${\em )} for some $\sigma > 1$, with respect to the eigenvalue sequence $\{\lambda_j\}_{j \in \mathbb{N}}$ of $P(x,D_x)$.
\end{theorem}

We divide the proof of Theorem~\ref{thm_GS} into two parts: sufficiency and necessity, which are treated in the next two subsections.

\subsection{Proof of Theorem \ref{thm_GS}: Sufficiency} \

Assume that $\boldsymbol{a}_0$ satisfies the Diophantine condition \eqref{DC} with respect to the sequence $\{\lambda_j\}_{j \in \mathbb{N}}$, and let $f \in \mathbb{E}_\mu^p$. As shown at the beginning of Section~\ref{sec-solvability}, if $u \in \mathscr{F}'_{\mu,p}$ satisfies $\mathbb{L}^p u = f$, then for every $(\tau, j) \in \mathbb{Z}^m \times \mathbb{N}$ we must have
\[
a_j(\tau) \wedge \widehat{u}_j(\tau) = \widehat{f}_j(\tau),
\]
where $a_j(\tau)$ is given by \eqref{ajtau}. Since $\mathbb{L}^{p+1} f = 0$ by assumption (as $f \in \mathbb{E}^p_\mu$), it follows that
\[
a_j(\tau) \wedge \widehat{f}_j(\tau) = 0, 
\]
for all \((\tau, j) \in \mathbb{Z}^m \times \mathbb{N}.\)

To separate resonant and non-resonant modes, let $S \subset \mathbb{N}$ be arbitrary. Following the strategy in \cite[Section 4]{BerCorPet96}, define the subspace
\[
\mathscr{F}'_{\mu,S} \coloneqq \left\{ u \in \mathscr{F}'_{\mu} \mid u(t,x) = \sum_{j \in S} u_j(t) \varphi_j(x) \right\},
\]
and let $\mathscr{F}'_{\mu,p,S} \coloneqq \textstyle\bigwedge^{p,0} \mathscr{F}'_{\mu,S}$ denote the space of $p$-currents of the form $u = \sum_{|J| = p} u_J(t,x) \, \mathrm{d}t_J$, where each $u_J \in \mathscr{F}'_{\mu,S}$. Analogously, we define the smooth space $\mathscr{F}_{\mu,p,S}$.

We then obtain the direct sum decompositions:
\[
\mathscr{F}_{\mu,p} = \mathscr{F}_{\mu,p,S} \oplus \mathscr{F}_{\mu,p,S^c}, \quad \mathscr{F}'_{\mu,p} = \mathscr{F}'_{\mu,p,S} \oplus \mathscr{F}'_{\mu,p,S^c},
\]
where $S^c = \mathbb{N} \setminus S$. Furthermore, the operator $\mathbb{L}^p$ respects these decompositions:
\[
\mathbb{L}^p(\mathscr{F}_{\mu,p,S}) \subset \mathscr{F}_{\mu,p+1,S}, \quad \mathbb{L}^p(\mathscr{F}'_{\mu,p,S}) \subset \mathscr{F}'_{\mu,p+1,S},
\]
and similarly for $S^c$.

We denote by $\mathbb{L}^p_S$ the restriction of $\mathbb{L}^p$ to $\mathscr{F}_{\mu,p,S}$.

\begin{definition}
	We say that the restricted operator $\mathbb{L}^p_S$ is \emph{globally solvable} if for every $f \in \mathbb{E}^p_{\mu,S} \coloneqq \mathbb{E}^p_\mu \cap \mathscr{F}_{\mu,p+1,S}$, there exists $u \in \mathscr{F}_{\mu,p,S}$ such that $\mathbb{L}^p_S u = f$.
\end{definition}

In particular, we have the decomposition
\[
\mathbb{E}^p_\mu = \mathbb{E}^p_{\mu,S} \oplus \mathbb{E}^p_{\mu,S^c},
\]
and the operator $\mathbb{L}^p$ is globally solvable if and only if both $\mathbb{L}^p_S$ and $\mathbb{L}^p_{S^c}$ are globally solvable.

Let us now choose $S = \mathcal{Z}$, where $\mathcal{Z} = \{j\in\mathbb{N}: \lambda_j\boldsymbol{a}_0\in\mathbb{Z}^m\} $ is the set of resonant indices defined in \eqref{set_Z}.

\begin{proposition}
	The operator
	\[
	\mathbb{L}^p_{\mathcal{Z}} :
	\mathscr{F}_{\mu,p,\mathcal{Z}}
	\longrightarrow
	\mathscr{F}_{\mu,p+1,\mathcal{Z}}
	\]
	is globally solvable for every $p=0,1,\dots,m-1$.
\end{proposition}

\begin{proof}
	Fix $p\in\{0,\dots,m-1\}$. For $u\in\mathscr{F}_{\mu,p,\mathcal{Z}}$ written as
	$u(t,x)=\sum_{j\in\mathcal{Z}}u_j(t)\varphi_j(x)$,
	consider the linear map
	\[
	\mathscr{I}_p u
	=
	\sum_{j\in\mathcal{Z}}u_j(t)\,e^{i\lambda_j\boldsymbol{a}_0\cdot t}\varphi_j(x).
	\]
	Since $\lambda_j\boldsymbol{a}_0\in\mathbb{Z}^m$ for $j\in\mathcal{Z}$,
	the exponential factor $e^{i\lambda_j\boldsymbol{a}_0\cdot t}$ is smooth and
	$2\pi$--periodic on $\mathbb{T}^m$.
	Using the characterization of $\mathscr{F}_{\mu,p}$ in terms of Fourier
	coefficients, it follows that $\mathscr{I}_p$ is a well-defined continuous
	isomorphism of $\mathscr{F}_{\mu,p,\mathcal{Z}}$ onto itself, with inverse
	given by multiplication by $e^{-i\lambda_j\boldsymbol{a}_0\cdot t}$.
	
	A direct computation shows that, for each $j\in\mathcal{Z}$,
	\[
	\mathrm{d}_t\!\left(e^{i\lambda_j\boldsymbol{a}_0\cdot t}u_j(t)\right)
	=
	e^{i\lambda_j\boldsymbol{a}_0\cdot t}
	\bigl(\mathrm{d}_t u_j(t)+ i\lambda_j \boldsymbol{a}_0\wedge u_j(t)\bigr),
	\]
	and therefore
	\[
	\mathscr{I}_{p+1}\circ\mathbb{L}^p_{\mathcal{Z}}\circ\mathscr{I}_p^{-1}
	=
	\mathrm{d}_t.
	\]
	Consequently, by the conjugation constructed above, the problem of global solvability for $\mathbb{L}^p_{\mathcal{Z}}$ reduces to the corresponding problem for the exterior derivative $\mathrm{d}_t$ acting on $\mathscr{F}_{\mu,p,\mathcal{Z}}$, using the same reasoning as in the proof of Proposition~\ref{equiv_L0}.

	For the operator $\mathrm{d}_t$, the space $\mathbb{E}^p_{\mu,\mathcal{Z}}$
	consists of all $(p+1)$--forms
	\[
	g(t,x)=\sum_{j\in\mathcal{Z}} g_j(t)\varphi_j(x)
	\in\mathscr{F}_{\mu,p+1,\mathcal{Z}}
	\]
	such that each coefficient $g_j(t)$ is an exact $(p+1)$--form on $\mathbb{T}^m$.
	Let such a $g$ be given. 	We seek $v\in\mathscr{F}_{\mu,p,\mathcal{Z}}$ satisfying
	\(\mathrm{d}_t v = g.\)
	
	Writing the Fourier expansions
	\[
	g_j(t)=\sum_{\tau\in\mathbb{Z}^m}\widehat g_j(\tau)e^{i\tau\cdot t},
	\qquad
	v_j(t)=\sum_{\tau\in\mathbb{Z}^m}\widehat v_j(\tau)e^{i\tau\cdot t},
	\]
	the equation $\mathrm{d}_t v_j=g_j$ is equivalent, for each $\tau\in\mathbb{Z}^m$, to
	\[
	\left(i\sum_{r=1}^m\tau_r\,\mathrm{d}t_r\right)\wedge \widehat v_j(\tau)
	=
	\widehat g_j(\tau).
	\]
		Since $g_j$ is exact, we necessarily have $\widehat g_j(0)=0$,
	and we set $\widehat v_j(0)=0$.
	For $\tau\neq0$, Lemma~\ref{sol_ext_alg} applies and yields a solution
	$\widehat v_j(\tau)\in\textstyle\bigwedge^p\mathbb{C}^m$.
	Moreover, choosing $r$ such that $\tau_r\neq0$, the lemma provides the estimate
	\[
	|\widehat v_j(\tau)|
	\le
	C_p\,\frac{|\widehat g_j(\tau)|}{|\tau_r|}
	\le
	C'_p\,|\widehat g_j(\tau)|,
	\]
	where the constants depend only on $p$ and $m$.
	
	Since $g\in\mathscr{F}_{\mu,p+1,\mathcal{Z}}$,
	its Fourier coefficients satisfy the defining exponential decay estimates.
	The bound above shows that the same type of estimates holds for
	$\widehat v_j(\tau)$.
	Therefore the series
	\[
	v(t,x)
	=
	\sum_{j\in\mathcal{Z}}
	\sum_{\tau\in\mathbb{Z}^m}
	\widehat v_j(\tau)\,e^{i\tau\cdot t}\varphi_j(x)
	\]
	defines an element of $\mathscr{F}_{\mu,p,\mathcal{Z}}$.
	
	By construction, $\mathrm{d}_t v=g$. 
	Conjugating back by $\mathscr{I}_p^{-1}$,
	we obtain a solution to $\mathbb{L}^p_{\mathcal{Z}}u=g$ in
	$\mathscr{F}_{\mu,p,\mathcal{Z}}$.
	Hence $\mathbb{L}^p_{\mathcal{Z}}$ is globally solvable.
\end{proof}

Now consider $h \in \mathbb{E}^p_{\mu,\mathcal{Z}^c}$. By Lemma~\ref{sol_ext_alg}, we can construct a solution $w$ to $\mathbb{L}^p w = h$ with Fourier coefficients
\[
\widehat{w}_j(\tau) = \frac{1}{i} \sum_{|J| = p+1} \sum_{r \in J} (-1)^{\mathrm{sign}(r, J)} \frac{(\widehat{h}_j(\tau))_J}{\tau_r + \lambda_j a_{r,0}} \, \mathrm{d}t_{J \setminus \{r\}},
\]
where $j \in \mathcal{Z}^c$, $\tau \in \mathbb{Z}^m$, and the index $r$ is chosen to maximize 
\[
|\tau_r + \lambda_j a_{r,0}| = \| \tau + \lambda_j \boldsymbol{a}_0 \|.
\]

Since $h \in \mathscr{F}_{\mu,p+1}$, there exist  \(C, \varepsilon_0>0\) such that
\[
\| \widehat{h}_j(\tau)_J \| \leq C \exp\left( - \varepsilon_0 \left( \|\tau\|^{1/\sigma} + j^{1/(2n\mu)} \right) \right).
\]

Combining this with the Diophantine condition \eqref{DC} (with $\varepsilon = \varepsilon_0/2$), we conclude that the series
\[
w(t,x) = \sum_{j \in \mathcal{Z}^c} \sum_{\tau \in \mathbb{Z}^m} \widehat{w}_j(\tau) e^{i \tau \cdot t} \varphi_j(x)
\]
defines a $p$-form in $\mathscr{F}_{\mu,p,\mathcal{Z}^c}$ satisfying $\mathbb{L}^p w = h$.

Finally, since $f = g + h$ with $g \in \mathbb{E}^p_{\mu,\mathcal{Z}}$ and $h \in \mathbb{E}^p_{\mu,\mathcal{Z}^c}$, and we have constructed $v \in \mathscr{F}_{\mu,p,\mathcal{Z}}$ and $w \in \mathscr{F}_{\mu,p,\mathcal{Z}^c}$ solving $\mathbb{L}^p v = g$ and $\mathbb{L}^p w = h$, we obtain a global solution $u = v + w \in \mathscr{F}_{\mu,p}$ of $\mathbb{L}^p u = f$.

	Hence $\mathbb{L}^p$ is globally solvable.

\subsection{Theorem \ref{thm_GS}: Necessity} \

Assume that the Diophantine condition \eqref{DC} fails. Then there exist
$\varepsilon_0>0$ and a sequence $\{(\tau_\ell,j_\ell)\}_{\ell\in\mathbb{N}}$ in $\mathbb{Z}^m\times\mathbb{N}$ such that
\begin{equation*}
	\|\tau_\ell\|+j_\ell\to\infty,\quad \ell\to\infty,
\end{equation*}
and
\begin{equation}\label{non_DC}
	0<\|\tau_\ell+\lambda_{j_\ell}\boldsymbol{a}_0\|
	\leq \exp\!\left(-\varepsilon_0\Big(\|\tau_\ell\|^{1/\sigma}
	+ j_\ell^{1/(2n\mu)}\Big)\right),
\end{equation}
for all $\ell\in\mathbb{N}$.

We define the growth sequence
\[
c_\ell := \exp\!\left(\frac{\varepsilon_0}{2}\Big(\|\tau_\ell\|^{1/\sigma} + j_\ell^{1/(2n\mu)}\Big)\right).
\]
For each $\ell \in \mathbb{N}$, let $u_\ell \in \mathscr{F}_{\mu,p}$ be defined by
\[
u_\ell(t,x) =
\begin{cases}
	c_\ell e^{i\tau_\ell \cdot t}\varphi_{j_\ell}(x), & p=0,\\[2mm]
	c_\ell e^{i\tau_\ell \cdot t}\varphi_{j_\ell}(x)\, \mathrm{d}t_2 \wedge \cdots \wedge \mathrm{d}t_{p+1}, & 1 \le p \le m-1.
\end{cases}
\]
While each term $u_\ell$ is smooth, the formal series
\[
u(t,x) := \sum_{\ell \in \mathbb{N}} u_\ell(t,x)
\]
does not define an element of $\mathscr{F}'_{\mu,p}$ due to the rapid growth of $c_\ell$, as guaranteed by Proposition~\ref{charac_full_fourier_ultradistr}.

Writing $\tau_\ell=(\tau_\ell^{(1)},\dots,\tau_\ell^{(m)})$, a direct
computation yields
\[
\mathbb{L}^0u_\ell
= i c_\ell e^{i\tau_\ell \cdot t} \varphi_{j_\ell}(x) \sum_{r=1}^m (\tau_\ell^{(r)}+a_{r,0}\lambda_{j_\ell})\,\mathrm{d}t_r,
\]
and, for $1\le p\le m-1$,
\[
\mathbb{L}^pu_\ell
= i c_\ell e^{i\tau_\ell \cdot t}\varphi_{j_\ell}(x) \sum_{r=1}^m (\tau_\ell^{(r)}+a_{r,0}\lambda_{j_\ell})\,
\mathrm{d}t_r\wedge\mathrm{d}t_2\wedge\cdots\wedge\mathrm{d}t_{p+1}.
\]

For each $r=1,\dots,m$, define
\[
f_r(t,x)
:= i\sum_{\ell\in\mathbb{N}}
c_\ell(\tau_\ell^{(r)}+a_{r,0}\lambda_{j_\ell})
e^{i\tau_\ell\cdot t}\varphi_{j_\ell}(x).
\]
By \eqref{non_DC} and the definition of $c_\ell$, we have
\[
|c_\ell(\tau_\ell^{(r)}+a_{r,0}\lambda_{j_\ell})|
\le \exp\!\left(-\frac{\varepsilon_0}{2}
\Big(\|\tau_\ell\|^{1/\sigma}+j_\ell^{1/(2n\mu)}\Big)\right),
\]
hence $f_r\in\mathscr{F}_\mu$ for all $r$.

Define the $(p+1)$-form $f$ by
\[
f(t,x)=
\begin{cases}
	\displaystyle \sum_{r=1}^m f_r(t,x)\,\mathrm{d}t_r, & p=0,\\[1mm]
	\displaystyle \sum_{r=1}^m f_r(t,x)\,
	\mathrm{d}t_r\wedge\mathrm{d}t_2\wedge\cdots\wedge\mathrm{d}t_{p+1},
	& 1\le p\le m-1.
\end{cases}
\]
Then $f=\sum_{\ell}\mathbb{L}^pu_\ell$ with convergence in
$\mathscr{F}_{\mu,p+1}$, so $\mathbb{L}^{p+1}f=0$.

Moreover, setting
\[
v_\ell(t)=
\begin{cases}
	c_\ell e^{i(\lambda_{j_\ell}\boldsymbol{a}_0+\tau_\ell)\cdot t},
	& p=0,\\[1mm]
	c_\ell e^{i(\lambda_{j_\ell}\boldsymbol{a}_0+\tau_\ell)\cdot t}\,
	\mathrm{d}t_2\wedge\cdots\wedge\mathrm{d}t_{p+1},
	& 1\le p\le m-1,
\end{cases}
\]
we obtain $\mathrm{d}_tv_\ell
= e^{i\lambda_{j_\ell}\boldsymbol{a}_0\cdot t}f_{j_\ell}(t)$, which implies
$f\in\mathbb{E}_\mu^p$.

\medskip
We now show that the equation $\mathbb{L}^p v=f$ admits no solution $v\in\mathscr{F}'_{\mu,p}$.
Since both $\widehat u_{j_\ell}(\tau_\ell)$ and
$\widehat v_{j_\ell}(\tau_\ell)$ are solutions of
\[
a_{j_\ell}(\tau_\ell)\wedge U = \widehat f_{j_\ell}(\tau_\ell),
\]
Lemma~\ref{sol_ext_alg} implies that there exists
$w_\ell\in\textstyle\bigwedge^{p-1}\mathbb{C}^m$ such that
\[
\widehat u_{j_\ell}(\tau_\ell)
= \widehat v_{j_\ell}(\tau_\ell)
+ a_{j_\ell}(\tau_\ell)\wedge w_\ell.
\]

If $p=0$, this leads to a contradiction, since $\textstyle\bigwedge^{-1} \mathbb{C}^m=\{0\}$ and
$u\notin\mathscr{F}'_{\mu,0}$.
Assume therefore $1\le p\le m-1$.
By reordering coordinates if necessary, we may assume
\[
|\tau_\ell^{(1)}+\lambda_{j_\ell}a_{1,0}|
= \|\tau_\ell+\lambda_{j_\ell}\boldsymbol{a}_0\|.
\]

Let
\[
\zeta_\ell :=
\widehat u_{j_\ell}(\tau_\ell)-\widehat v_{j_\ell}(\tau_\ell)
= a_{j_\ell}(\tau_\ell)\wedge w_\ell,
\qquad
w_\ell\in\textstyle\bigwedge^{p-1}\mathbb C^m,
\]
as given by Lemma~\ref{sol_ext_alg}.
Writing
\[
\zeta_\ell=\sum_{|K|=p}\zeta_K^{(\ell)}\,\mathrm dt_K,
\]
a detailed combinatorial analysis of the wedge product
$a_{j_\ell}(\tau_\ell)\wedge w_\ell$
(see Lemma~\ref{lem:combinatorial_identity} below)
yields the identity
\begin{equation}\label{zeta_1}
	\zeta_{(1)}^{(\ell)}
	=
	\sum_{s=1}^{p+1}
	(-1)^s
	\frac{\tau_\ell^{(s)}+\lambda_{j_\ell}a_{s,0}}
	{\tau_\ell^{(1)}+\lambda_{j_\ell}a_{1,0}}
	\zeta_{(s)}^{(\ell)}.
\end{equation}

Using this identity and the growth estimates satisfied by
$\widehat v_{j_\ell}(\tau_\ell)$, one obtains
\[
|c_\ell|
\le
C_\varepsilon
\exp\!\left(
\varepsilon\bigl(\|\tau_\ell\|^{1/\sigma}+j_\ell^{1/(2n\mu)}\bigr)
\right),
\]
for every $\varepsilon>0$.
Choosing $\varepsilon=\varepsilon_0/4$ contradicts the definition of $c_\ell$,
hence such a distributional solution $v$ cannot exist.

\begin{lemma}\label{lem:combinatorial_identity}
	Let $1 \leq p \leq m-1$ and let
	\[
	\mathscr L
	=
	i\sum_{r=1}^m \xi_r\,\mathrm dt_r
	\in \textstyle\bigwedge^1 \mathbb C^m,
	\quad \mbox{ with } \xi_1 \neq 0.
	\]
	Assume that $\zeta \in \textstyle\bigwedge^p \mathbb C^m$ is of the form
	\[
	\zeta = \mathscr L \wedge w,
	\quad \mbox{ for some }
	w \in \textstyle\bigwedge^{p-1} \mathbb C^m.
	\]
	Writing
	\[
	w = \sum_{|J|=p-1} w_J \, \mathrm dt_J,
	\qquad
	\zeta = \sum_{|K|=p} \zeta_K \, \mathrm dt_K,
	\]
	and defining, for $s=1,\dots,p+1$,
	\[
	\mathrm dt_{(s)} := \mathrm dt_1 \wedge \cdots \wedge \widehat{\mathrm dt_s}
	\wedge \cdots \wedge \mathrm dt_{p+1},
	\qquad
	\zeta_{(s)} := \zeta_{(1,\dots,\widehat s,\dots,p+1)},
	\]
	the following identity holds:
	\begin{equation}\label{zeta_identity}
		\zeta_{(1)}
		=
		\sum_{s=1}^{p+1}
		(-1)^s
		\frac{\xi_s}{\xi_1}
		\zeta_{(s)}.
	\end{equation}
\end{lemma}
\begin{proof}
	By assumption,
	\[
	\zeta =	\mathscr L \wedge w = 	\left( i\sum_{r=1}^m \xi_r\,\mathrm dt_r \right)
	\wedge 	\left( \sum_{|J|=p-1} w_J\,\mathrm dt_J \right).
	\]
	Expanding the wedge product, we obtain
	\[
	\zeta 	= 	i\sum_{r=1}^m \sum_{|J|=p-1}  \xi_r w_J \, 	\mathrm dt_r \wedge \mathrm dt_J.
	\]
	
	Let $K=(k_1,\dots,k_p)$ be a strictly increasing multi-index.
	The coefficient $\zeta_K$ of $\mathrm dt_K$ arises precisely from the terms
	for which $r\in K$ and $J=K\setminus\{r\}$.
	In this case,
	\[
	\mathrm dt_r \wedge \mathrm dt_{K\setminus\{r\}} 	=
	(-1)^{\mathrm{pos}(r,K)-1}\,\mathrm dt_K,
	\]
	where $\mathrm{pos}(r,K)$ denotes the position of $r$ inside $K$.
	Therefore,
	\[
	\zeta_K 	= 	i\sum_{r\in K} (-1)^{\mathrm{pos}(r,K)-1} \xi_r\, w_{K\setminus\{r\}}.
	\]
	
	We now restrict attention to the family of multi-indices
	\[
	K_s := (1,\dots,\widehat s,\dots,p+1), 	\qquad s=1,\dots,p+1,
	\]
	and denote $\zeta_{(s)}:=\zeta_{K_s}$.
	For such indices, the elements of $K_s$ are exactly the integers
	$\{1,\dots,p+1\}\setminus\{s\}$.
	Hence, for $r\in K_s$, there are two cases.
	
	If $r<s$, then $r$ occupies position $r$ in $K_s$, and thus
	\[
	(-1)^{\mathrm{pos}(r,K_s)-1}=(-1)^{r-1}.
	\]
	If $r>s$, then $r$ occupies position $r-1$ in $K_s$, yielding
	\[
	(-1)^{\mathrm{pos}(r,K_s)-1}=(-1)^{r}.
	\]
	Consequently,
	\[
	\zeta_{(s)} = \sum_{1\le r<s} i(-1)^{r-1}\xi_r\, w_{(r,s)} 	+
	\sum_{j<r\le p+1} i(-1)^{r}\xi_r\, w_{(s,r)}.
	\]
	
	We now consider the linear combination
	\[
	S := \sum_{s=1}^{p+1} (-1)^s \xi_s \zeta_{(s)}.
	\]
	Substituting the expression above and rearranging the sums, we obtain
	\[
	S = i\sum_{r<s} \Bigl[ (-1)^{s+r-1} + (-1)^{s+r} \Bigr] \xi_s\xi_r\, w_{(r,s)}. 
	\]
	Since
	\[
	(-1)^{s+r-1} + (-1)^{s+r} = 0,
	\]
	all terms cancel pairwise, and therefore $S=0$.
	
	Thus,
	\[
	\sum_{s=1}^{p+1} (-1)^s \xi_s \zeta_{(s)} = 0.
	\]
	Using the assumption $\xi_1\neq 0$, we can solve this relation for
	$\zeta_{(1)}$, obtaining
	\[
	\zeta_{(1)} = \sum_{s=1}^{p+1} (-1)^s \frac{\xi_s}{\xi_1} \zeta_{(s)},
	\]
	which proves \eqref{zeta_identity}.
\end{proof}

\begin{remark}
	In several related works, such as
	\cite{ALMed24,AviMed25,BerPet99jmaa},
	the authors consider a weaker notion of global solvability.
	In the present setting, this weaker requirement can be formulated as
	\begin{equation}\label{weak_GS}
		f \in \mathbb{E}_\mu^p
		\quad \Longrightarrow \quad
		\exists\, u \in \mathscr{F}'_{\mu,p}
		\ \text{such that}\ 
		\mathbb{L}^p u = f .
	\end{equation}
	
	By definition, global solvability of $\mathbb{L}^p$ in the sense adopted in
	this paper (solvability with solutions in $\mathscr{F}_{\mu,p}$)
	clearly implies \eqref{weak_GS}.
	The significance of Theorem~\ref{thm_GS} lies in the fact that, for the class
	of operators considered here, these two notions of solvability actually
	coincide.
	
	Indeed, in the proof of the sufficiency part of Theorem~\ref{thm_GS} we showed
	that for every $f \in \mathbb{E}_\mu^p$ there exists a solution
	$u \in \mathscr{F}_{\mu,p}$, which in particular implies \eqref{weak_GS}.
	Conversely, in the proof of the necessity, we constructed an element
	$f \in \mathbb{E}_\mu^p$ for which no solution
	$u \in \mathscr{F}'_{\mu,p}$ to $\mathbb{L}^p u = f$ exists, showing that
	\eqref{weak_GS} fails whenever $\mathbb{L}^p$ is not globally solvable in our
	sense.
	
	This equivalence shows that our choice of definition is natural in the
	present framework.
	Moreover, it has the additional advantage of allowing the use of
	\cite[Lemma~2.2]{Araujo2017} in the proof of
	Proposition~\ref{equiv_L0}, which relies on the DFS structure of the spaces
	$\mathscr{F}_{\mu,p}$.
\end{remark}

\section{Global Hypoellipticity}\label{sec-hypoellipticity}

We now turn to the problem of global hypoellipticity for the operators
$\mathbb{L}^p$.
As in the case of global solvability, the behavior of the complex turns out
to be sharply different according to the degree $p$.

For $p=0$, a complete characterization is already available in the
literature.
More precisely, Propositions~3.1 and~3.2 of \cite{AviCapKir25} yield the
following result.

\begin{theorem}\label{hypo0}
	The operator $\mathbb{L}^0$ is globally hypoelliptic if and only if
	the set $\mathcal{Z}$ defined in \eqref{set_Z} is finite and
	$\boldsymbol{a}_0$ satisfies the Diophantine condition \eqref{DC}
	with respect to the spectrum $\{\lambda_j\}_{j\in\mathbb{N}}$
	of $P(x,D_x)$.
\end{theorem}

In contrast, the situation changes drastically for higher degrees.
The presence of nontrivial differential forms introduces additional
obstructions that prevent global hypoellipticity, regardless of any
Diophantine condition.
Adapting the strategy developed in \cite{DatMez20} to our setting, we prove
that global hypoellipticity fails for all $p \ge 1$.

\begin{theorem}\label{hypo>0}
	If $p \ge 1$, then the operator $\mathbb{L}^p$ is not globally
	hypoelliptic.
\end{theorem}

\begin{proof}
	Recall the definition of the set $\mathcal{Z}$ in \eqref{set_Z}.
	We distinguish two cases according to its cardinality.
	
	\smallskip
	\noindent\emph{Case 1: $\mathcal{Z}$ is infinite.}
	In this situation, consider the distribution
	\[
	v(t,x)
	=
	\sum_{j\in\mathcal{Z}}
	e^{-i(\lambda_j\boldsymbol{a}_0)\cdot t}\,\varphi_j(x),
	\]
	which belongs to $\mathscr{F}'_\mu$ but not to $\mathscr{F}_\mu$.
	Let $\eta$ be any nonzero $p$-form with constant coefficients on
	$\mathbb{T}^m$.
	Then $u := v\,\eta$ defines an element of
	$\mathscr{F}'_{\mu,p}\setminus\mathscr{F}_{\mu,p}$.
	A direct computation shows that $\mathbb{L}^p u = 0$,
	which proves that $\mathbb{L}^p$ is not globally hypoelliptic.
	
	\smallskip
	\noindent\emph{Case 2: $\mathcal{Z}$ is finite.}
	Let $\mathcal{Z}^c = \mathbb{N}\setminus\mathcal{Z}$.
	For each $(\tau,j)\in\mathbb{Z}^m\times\mathcal{Z}^c$, consider the
	$1$-form $a_j(\tau)$ defined in \eqref{ajtau}.
	Choose a $(p-1)$-form $\eta_j(\tau)$ with constant coefficients on
	$\mathbb{T}^m$ such that
	$a_j(\tau)\wedge\eta_j(\tau)\neq 0$.
	Writing
	\[
	a_j(\tau)\wedge\eta_j(\tau)
	=
	\sum_{|J|=p}\theta_j(\tau)_J\,\mathrm{d}t_J,
	\]
	define
	\[
	C_{\tau,j}
	:=
	\max\{|\theta_j(\tau)_J|:\ |J|=p\}
	\,>\,0.
	\]
	
	We then consider the $p$-form
	\[
	u(t,x)
	=
	\sum_{(\tau,j)\in\mathbb{Z}^m\times\mathcal{Z}^c}
	\frac{1}{C_{\tau,j}}\,
	e^{i\tau\cdot t}\,\varphi_j(x)\,
	a_j(\tau)\wedge\eta_j(\tau).
	\]
	By construction, $u$ belongs to
	$\mathscr{F}'_{\mu,p}\setminus\mathscr{F}_{\mu,p}$.
	Moreover, since $a_j(\tau)\wedge a_j(\tau)=0$, we have
	$\mathbb{L}^p u = 0$.
	
	In both cases, we obtain a distributional solution $u\notin
	\mathscr{F}_{\mu,p}$ of the homogeneous equation
	$\mathbb{L}^p u = 0$, which proves that $\mathbb{L}^p$ cannot be
	globally hypoelliptic for $p\ge1$.
\end{proof}


\bibliographystyle{plain} 
\bibliography{references2} 

\end{document}